\documentclass[final]{amsart}

\usepackage{ppack}

\title[Multiplication polynomials for EC over finite local rings]{Multiplication polynomials for elliptic curves over finite local rings}

\author{Riccardo Invernizzi}
\address{Department of Computer Science, \\
	KU Leuven\\
	Celestijnenlaan 200A, Leuven, Belgium}
\thanks{R.I. was supported by IUSS Pavia}

\email[R.~Invernizzi]{riccardo.invernizzi@kuleuven.be}
\urladdr[R. Invernizzi]{https://orcid.org/0000-0002-2271-6822} 

\author{Daniele Taufer}
\address{Department of Computer Science, \\
	KU Leuven\\
	Celestijnenlaan 200A, Leuven, Belgium}
\thanks{D.T. was supported by the Research Foundation - Flanders (FWO), project 12ZZC23N}
\email[D.~Taufer]{daniele.taufer@kuleuven.be}
\urladdr[D. Taufer]{https://orcid.org/0000-0003-3402-4863}

\begin{document}


\begin{abstract}
	For a given elliptic curve $E$ over a finite local ring, we denote by $E^{\infty}$ its subgroup at infinity.
	Every point $P \in E^{\infty}$ can be described solely in terms of its $x$-coordinate $P_x$, which can be therefore used to parameterize all its multiples $nP$.
	We refer to the coefficient of $(P_x)^i$ in the parameterization of $(nP)_x$ as the $i$-th multiplication polynomial.
	We show that this coefficient is a degree-$i$ rational polynomial without a constant term in $n$.
	We also prove that no primes greater than $i$ may appear in the denominators of its terms. 
	As a consequence, for every finite field $\F_q$ and any $k\in\N^*$, we prescribe the group structure of a generic elliptic curve defined over $\F_q[X]/(X^k)$, and we show that their ECDLP on $E^{\infty}$ may be efficiently solved.
\end{abstract}
\maketitle

\section{Introduction}
Elliptic curves are fascinating objects that have been attracting considerable attention from several different fields, such as number theory \cite{Fermat} and algebraic cryptography \cite{Miller, Koblitz}.
One of the key features of these objects is the fact that they have been proven to be abelian varieties, and as such, they are naturally endowed with a group structure.

Remarkably, the points of such groups can be efficiently handled, but their algebra may almost never be read from their coordinate representation.
Thence, the entries of the multiples of a given point usually look random and therefore provide no information about the underlying group operation.
This feature has been heavily exploited to design discrete logarithm-based cryptosystems, such as key agreement \cite{DH}, signature schemes \cite{sign}, and pseudorandom number generators \cite{Shparlinski}.

However, point multiplication may be read from point coordinates in a few cases.
For instance, the algebra on the group at infinity of non-canonical lifting of anomalous curves has been employed for efficiently solving the discrete logarithm problem on these curves \cite{sala20znz,SatohAraki,Smart}.

In this work, we adopt a novel approach to address the group at infinity $E^{\infty}$ of elliptic curves $E$ defined over finite local rings $(\rcal, \m)$, which include a wide portion of curves with practical interest \cite{sala23surv}.
First, we provide an efficient description of the addition law for these points (Proposition \ref{prop:shortsum}), and we show that every point $P \in E^{\infty}$ may be represented as $P = \big(X:1:\ff(X)\big)$, for a prescribed polynomial $\ff \in \rcal[x]$ (Proposition \ref{prop:zfx}).
Therefore, one can symbolically compute the $n$-th multiple of $P$ as $nP = \big( (nP)_x : 1 : \ff((nP)_x) \big)$.
We define the \emph{multiplication polynomials} $\psi_i$ as the maps sending $n \in \N$ to the coefficient of $X^i$ in $(nP)_x$, namely for every $n \in \N$ we have
\[ (nP)_x = \psi_1(n) X +  \psi_2(n) X^2 + \dots +  \psi_{k-1}(n) X^{k-1}, \]
where $k$ is the minimal integer such that $\m^k = (0)$.
These objects are, a priori, just functions of $n$.
However, we prove that every $\psi_i(n)$ is actually a polynomial of degree $i$ over the rationals and the curve coefficients, and that we have $n | \psi_i(n)$ (Theorem \ref{teo:psi_poli}).
Furthermore, we show that the prime divisors appearing in the denominators of $\psi_i(n)$ may never be larger than $i$ (Theorem \ref{teo:psi_int}).

These facts prescribe general arithmetic properties of the scalar multiplication in $E^{\infty}$, especially when the considered scalar is the characteristic of the residue field $\rcal/\m$ (Corollary \ref{cor:psi_i_zero}).
We present an application of these results for determining the group structure of elliptic curves arising over $\rcal = \F_q[x]/(x^k)$.
This was an open problem \cite[Section 11]{sala23surv}, which we completely solve for \emph{generic} elliptic curves, namely all elliptic curves but those satisfying special conditions (Theorem \ref{teo:grp_struct_c1} and Corollary \ref{cor:grp_struct}).

We also discuss the structure of $E^\infty$ in the remaining cases, providing their classification under three broad conditions on the curve coefficients (Theorem \ref{teo:grp_struct_c1}).
We prove that these conditions always hold for rings of characteristic $2$ or $3$, and we computationally verify them for all the elliptic curves within the reach of our calculators.

Finally, in Section \ref{subs:ecdlp} we observe that solving the elliptic curve discrete logarithm problem over these special rings is not substantially harder than the same problem over their residue fields, and discuss the potential cryptographic implications.

\subsection{Paper organization}
In Section \ref{sec:Notation}, we recall the known definitions and results that we employ in the paper.
Efficient computation of the addition law is presented in Section \ref{sec:Addition}, while the standard form of points at infinity is presented in Section \ref{sec:pointsO}.
Section \ref{sec:MultPoly} is devoted to the definition of multiplication polynomials and to establishing their main properties.
In Section \ref{sec:GrpStructure}, the previous results are applied to determine the group of elliptic curves over $\F_q[x]/(x^k)$.
Symbolic verification and computational tests can be found at \cite{repo}.
\section{Notation and standard results} \label{sec:Notation}
Let $\rcal$ be a finite local ring, whose maximal ideal will be denoted by $\m$.
Since it is finite, its residue field $\rcal / \m \simeq \F_q$ is a finite field, and its size $\#\rcal$ is a power of $q$.
Moreover, there is $k \in \N$ such that $\m^k = 0$.
The minimal such $k$ will be then referred to as the nilpotence degree of $\rcal$.
Hence, every element $r \in \m$ is nilpotent, while $\rcal \backslash \m = \rcal^*$, i.e. every non-nilpotent element is invertible. 

The projective plane over $\rcal$ is the set of classes $(X:Y:Z)$ representing primitive triples $(X,Y,Z)$ modulo the action of $\rcal^*$ given by the component-wise multiplication.
In other terms, the elements of $\P^2(\rcal)$ are the projective points $(X:Y:Z)$ with $\langle X,Y,Z \rangle = \langle 1 \rangle = \rcal$, identified by the equivalence relation
\[\begin{gathered}
    (X_1:Y_1:Z_1)=(X_2:Y_2:Z_2) \ \textnormal{ if and only if} \\
    X_1Y_2-X_2Y_1 = X_1Z_2-X_2Z_1 = Y_1Z_2-Y_2Z_1 = 0.
\end{gathered}\]

An elliptic curve $E$ over $\rcal$ is the set of plane projective points satisfying a non-singular Weierstrass equation over $\rcal$, namely
\[ y^2z + a_1xyz + a_3yz^2 = x^3 + a_2x^2z + a_4xz^2 + a_6z^3. \]
The non-singularity condition amounts to having an invertible prescribed polynomial combination $\Dt_E$ of the coefficients $a_i$, i.e. $\Dt_E \in \rcal^*$.
The precise definition of such $\Dt_E$ may be found in \cite[Section III.1]{silverman09}, with the following minor correction: 
\begin{itemize}
    \item $b_2 = a_1^2 + 4a_2$ instead of $b_2 = a_1^2 + 4a_4$.
\end{itemize}

These objects are known to have a group structure defined via the bihomogeneous polynomials of bidegree $(2,2)$ explicitly given in \cite{bosma95}, modulo the corrections, reported by \cite{bern2011}, of two minor typos:
\begin{itemize}
    \item in $X_3^{(2)}$, write the term $a_3a_4(-2X_1Z_2 - X_2Z_1)X_2Z_1$ in place of $a_3a_4(X_1Z_2 - 2X_2Z_1)X_2Z_1$,
    \item in $Y_3^{(2)}$, write $-(3a_2a_6 - a_4^2)(-2X_1Z_2 - X_2Z_1)X_2Z_1$ instead of $-(3a_2a_6 - a_4^2)(X_1Z_2 + X_2Z_1)(X_1Z_2 - X_2Z1)$.
\end{itemize}
We will denote the law associated to the point $(0:0:1)$ as $+_{(0:0:1)}$, and the one associated to $(0:1:0)$ as $+_{(0:1:0)}$. 
We recall that a pair of points is exceptional for $+_{(0:0:1)}$ if and only if the $z$-coordinate of the sum is a zero divisor, while it is exceptional for $+_{(0:1:0)}$ if and only if the $y$-coordinate of the sum is a zero divisor. 
Therefore, $+_{(0:0:1)}$ and $+_{(0:1:0)}$ form a complete system of addition laws for any elliptic curve, whose combined action will be simply denoted as $+$.
The identity of these groups is $\ocal = (0:1:0)$, and the inverse of a given point $(X:Y:Z)$ is given by
\[ -(X:Y:Z) = (X:-Y-a_1X-a_3Z:Z). \]

Given a point $P = (X:Y:Z) \in \P(\rcal)$ of an elliptic curve defined over $\rcal$, it may be uniquely represented as
\[P = \begin{cases} (X \cdot Z^{-1} : Y \cdot Z^{-1} : 1), & \textnormal{if } Z \in \rcal^*, \\
(X \cdot Y^{-1} : 1 : Z \cdot Y^{-1}), & \textnormal{otherwise.}
\end{cases}\]
The points that admit a representation of the first type are called \emph{affine}, while the others are called \emph{at infinity}.
Those points lie above the affine (resp. at infinity) points of the underlying curve defined over the residue field $\F_q \simeq \rcal / \m$, via the componentwise canonical projection
\begin{equation*}
    \pi : E(R) \rightarrow E(\F_q), \quad (X:Y:Z) \mapsto ([X] : [Y] : [Z]).
\end{equation*}
It is well known that since the addition laws are polynomial in the coordinates of the points, then $\pi$ is a surjective group homomorphism \cite[Sec. 4]{lenstra86}. 
The group of points at infinity of an elliptic curve $E$ is its subgroup denoted by $E^\infty = \pi^{-1}(\ocal)$.
 \section{Addition Law} \label{sec:Addition}
Computing the point addition on elliptic curves defined over a ring usually requires computing a valid linear combination of the triples obtained by a complete system of addition laws \cite[Sec. 3]{lenstra86}.
However, over a finite local ring $\rcal$ this computation is simpler, as the sum is always directly computed by $+_{(0:0:1)}$ or by $+_{(0:1:0)}$.
\begin{proposition}
    Let $P_1, P_2$ two points of an elliptic curve defined over $\rcal$. Then $P_1+P_2$ is always computed by $P_1 +_{(0:0:1)} P_2$ or by $P_1 +_{(0:1:0)} P_2$.
\end{proposition}
\begin{proof}
    Since these addition laws form a complete system, thanks to \cite[Sec. 3]{lenstra86} at least one between $P_1 +_{(0:0:1)} P_2$ and $P_1 +_{(0:1:0)} P_2$ is guaranteed to contain a non-nilpotent entry. Since non-nilpotent elements of $\rcal$ are invertible, then one entry is a unit and therefore we have a valid projective point.
\end{proof}

\begin{notation*} From now on, we will denote the points symbolically, namely their coordinates will be regarded as variables instead of elements of $\rcal$.
With this slight abuse of notation, we shorten and simplify the statements, and every result we prove holds regardless of the specific point one starts from.
As an instance, given a point $P = (X:Y:Z) \in \P^2(\rcal)$, we will consider $XYZ$ as a degree-$3$ polynomial rather than an element of $\rcal$.
\end{notation*}

The following proposition gives an elegant and efficient way of computing $+_{(0:1:0)}$.
\begin{proposition}
    \label{prop:shortsum}
    For $i \in \{1,2\}$, let $P_i = (X_i:Y_i:Z_i) \in \P^2(\rcal)$ be two projective points. Let also
    \[ (X_3,Y_3,Z_3) = P_1 +_{(0:1:0)} P_2, \]
    and
    \begin{equation*}
        \begin{gathered}
            g_1 = X_2(a_1X_1 + a_3Z_1 + Y_1) + X_1Y_2, \\
            g_2 = Z_2(a_1X_1 + a_3Z_1 + Y_1) + Z_1Y_2.
        \end{gathered}
    \end{equation*}
    Then there exist four bihomogenous polynomials 
    \[ H_1,\dots,H_4 \in \rcal[X_i,Y_i,Z_i]_{i \in \{1,2\}} \]
    of bidegree $(1,1)$ such that
    \begin{equation*}
        \begin{gathered}
            X_3 = g_1H_1 + g_2H_2,\\
            Z_3 = g_1H_3 + g_2H_4,\\
            Y_3 = H_1H_4 - H_2H_3.
        \end{gathered}
    \end{equation*}
\end{proposition}
\begin{proof}
    Straightforward computation. The explicit polynomials $H_i$ and the actual formal verification may be found in the Appendix (Proposition \ref{prop:shortsum_calc}).
\end{proof}
\begin{remark}
    As we will show shortly, in this work we will always consider points that are not exceptional for the second addition law. These reduced formulas will be of great help both in terms of understanding the sum and computational speed.
\end{remark}

\begin{remark}
    Proposition \ref{prop:shortsum} is the natural generalization of \cite[Lemma 2.1]{salataufer} and works for every elliptic curve with an extended Weierstrass model over any admissible ring, regardless of the characteristic.
\end{remark}

\section{Points over \texorpdfstring{$\ocal$}{O}} \label{sec:pointsO}
In this section, we describe a convenient way of representing points in $E^{\infty}$.
We will establish results that hold symbolically for all such points, i.e. they hold for every specialization of their entries in $\rcal$.
\begin{remark}
    \label{rmk:use+2}
    Let $P = (X:Y:Z) \in E^{\infty}$. Since by definition $\pi(P) = (0:1:0)$, its standard form will always be $P = (X:1:Z)$ with $X,Z \in \m$. This also implies that $+_{(0:1:0)}$ is always valid over $E^{\infty}$. From now when we add two points we are implicitly using this addition law.
\end{remark}

\begin{notation*}
Given a projective point $P \in \P^2(\rcal)$ at infinity, we will denote by $P_x$ (resp. $P_z$) the $x$-coordinate (resp. $z$-coordinate) of its standard form $(P_x:1:P_z)$.
\end{notation*}

\begin{prop}
    \label{prop:zfx}
    Let $E$ be an elliptic curve over $\rcal$.
    There is a polynomial $\ff \in \rcal[x]$ of degree strictly lower than the nilpotence degree of $\rcal$, such that for every $P \in E^{\infty}$ we have $P = \big(P_x:1:\ff(P_x)\big)$.
    Moreover, $x^3 | \ff(x)$.
\end{prop}
\begin{proof}
    This is the same idea of \cite[Prop. 11]{sala20znz}. Every point in $E^{\infty}$ satisfies
    $$z = x^3 -a_1xz + a_2x^2z - a_3z^2 + a_4xz^2 + a_6z^3,$$
    or $z = \ff(x, z)$. We can hence replace $z$ with $\ff(x, z)$ on the right side obtaining
    $$z = \ff\big(x, \ff(x, ... \ff(x,z) ... )\big).$$
    In this way the degree in $\rcal[x,z]$ of every monomial containing $z$ increases every time, and since in this ring $z^k = x^k = 0$ (because $P_x,P_z \in \m$), after a finite number of substitutions we are left with $z = \ff(x)$.
    The computation of the explicit expression of $\ff$ truncated to small exponents can be found in the Appendix (Proposition \ref{prop:zfx_calc}), from which one easily observes that $x^3 | \ff(x)$.
\end{proof}
By Proposition \ref{prop:zfx} we see that a generic point $P \in E^{\infty}$ is entirely determined by $P_x$ and the coefficients of the Weierstrass equation of $E$.
Up to now, we assumed to be working with a fixed curve $E$. However, since the statement of Proposition \ref{prop:zfx} holds independently from the chosen curve $E$, we can let $E$ and hence its coefficients $a_i$ vary.
In this way, we get the multivariate function 
\[ z = \ff(x, a_i) \in \F_q[a_1, \dots, a_6][x]. \]
\begin{lemma}
    \label{lem:y_inv}
    Given any curve $E$ over $\rcal$ and three points $P,Q,R \in E^\infty$ such that $P+Q=R$, we have 
    \[ R_x \in \langle P_x, Q_x \rangle \subset \Z[a_1,\dots,a_6][P_x,Q_x].\]
\end{lemma}
\begin{proof}
    It follows from a direct inspection of the addition formulae. Further details can be found in Proposition \ref{prop:sommemodulop2}, applied with $P_1=P$, $P_2=Q$ and $P_3 = R$. With that notation, both $I_P$ and all the other terms of $R_x$ are clearly contained in $\langle P_x, Q_x \rangle$.
\end{proof}
\begin{remark}
    \label{rmk:imultiplidipsonobelli}
    When both $P$ and $Q$ are a multiple of a same point $(X:1:Z)$, by Lemma \ref{lem:y_inv} we have $R_x \in \langle X \rangle$.
\end{remark}

\section{Multiplication Polynomials} \label{sec:MultPoly}
\begin{lemma}
    \label{lem:psi_esun}
    Given an elliptic curve $E$ over $\rcal$, for every $n \in \N$ there are uniquely defined coefficients $\psi_1(n),...,\psi_{k-1}(n) \in \rcal$ such that for every symbolic $P = \big(X:1:\ff(X)\big) \in E^\infty$ we have
    \[ (nP)_x = \Sum_{i=1}^{k-1} \psi_i(n)X^i. \]
\end{lemma}
\begin{proof}
    From Lemma \ref{lem:y_inv} we know that $(nP)_x$ is a polynomial function of $X$ without constant term, which proves the existence. 
    As for uniqueness, let us assume that we can also write
    $$ (nP)_x = \Sum_{i=1}^{k-1} \varphi_i(n)X^i.$$
    This implies
    $$ 0 = \Sum_{i=1}^{k-1} \left[ \psi_i(n)-\varphi_i(n)\right]X^i$$
    and since the $X^i$ are a basis for polynomials in $X$, this shows that $\psi_i(n)=\varphi_i(n)$ for every $1 \leq i \leq k-1$.
\end{proof}

\begin{remark}
	The coefficients $\psi_i(n)$ depend on the coefficients $a_i$ of the given elliptic curve, therefore they may also be regarded as functions $\psi_i(n,a_i)$.
\end{remark}

\begin{defi}
	\label{def:psi}
	For every $1 \leq i \leq k-1$ we define the \emph{$i$-th multiplication polynomial} $\psi_i$ as the unique function over $\N$ such that $\psi_i(n)$ is the coefficient of $X^i$ in $(nP)_x$, as determined in Lemma \ref{lem:psi_esun}.
\end{defi}
At this stage, it may not be clear that they are actual polynomials, as it will be proved in Theorem \ref{teo:psi_poli}.
\begin{remark} \label{rmk:psii1}
	By definition, it holds $\psi_i(1) = 0$ for all $i\geq 2$.
\end{remark}
\begin{remark}
	\label{rem:deg_influence}
	Since the addition law is polynomial in the entries of the addenda, computing the coefficient of $X^i$ in $(nP)_x$ never requires computing coefficients of $X^j$ with $j>i$.
	For this reason, we may perform every computation of $\psi_i(n)$ modulo $X^{i+1}$, as if the nilpotence of $\rcal$ was $i+1$.
\end{remark}
\begin{lemma}
	\label{lem:psi1}
	With the above notation, we have
	$$\psi_1(n) = n.$$
\end{lemma}
\begin{proof}
	Thanks to Remark \ref{rem:deg_influence} we may assume $k=2$. This implies that for every $P \in E^{\infty}$ we have $P_z = \ff(P_x) = 0$, and the addition between two such points becomes 
	$$(X_1:1:0) + (X_2:1:0) = (X_1+X_2:1:0).$$ 
	In this case, the curve addition simply corresponds to the standard ring addition in the first entry, then $nP = (n P_x : 1 : 0).$
\end{proof}
\begin{lemma}
	\label{lem:psi2}
	With the above notation, we have
	$$\psi_2(n) = \binom{n}{2}a_1. $$
\end{lemma}
\begin{proof}
	By Remark \ref{rem:deg_influence} we can assume $k=3$. Again, this implies that for every $P \in E^{\infty}$ we have $P_z = 0$, hence $P = (X:1:0)$.
	We find $\psi_2(n)$ recursively, by computing 
	\[ nP = (X:1:0) + \big( \psi_1(n-1)X + \psi_2(n-1)X^2 : 1 : 0 \big). \]
	Performing this addition, we obtain
	\[ nP = \big( (1+\psi_1(n-1)) X + ( a_1 + 2 a_1 \psi_1(n-1) + \psi_2(n-1) ) X^2 : 1 + a_1 X : 0 \big). \]
	The inverse of its $y$-coordinate is $1-a_1 X-a_1^2X^2$, therefore
	\[ nP = \big( (1+\psi_1(n-1)) X + ( a_1\psi_1(n-1) + \psi_2(n-1) ) X^2 : 1 : 0 \big). \]
	Hence, we have
	\begin{equation*}
		\psi_2(n) = a_1\psi_1(n-1) + \psi_2(n-1),
	\end{equation*}
	which leads to the recurrence relation
	$$\psi_2(n) - \psi_2(n-1) = a_1 \psi_1(n-1) = a_1 (n-1),$$
    where the last equality follows from Lemma \ref{lem:psi1}.
	Since $\psi_2(1) = 0$, by Remark \ref{rmk:psii1}, we get 
	$$ \psi_2(n) = \Sum_{m=1}^{n-1} \big(\psi_2(m+1) - \psi_2(m) \big) =  a_1 \frac{n(n-1)}{2}, $$
	which concludes the proof.
\end{proof}
The explicit computation of $\psi_i(n)$ becomes increasingly harder for larger values of $i$. However, the technique used in Lemma \ref{lem:psi2} can be used to infer useful properties about these objects.
To prove them, we need the following technical lemma.

\begin{lemma}
	\label{lem:contirognosi}
	Let $\Z[a_1,\dots,a_6][\bt_1, \dots, \bt_{k-1}]$ be the graded ring of weights $\deg_\bt(\bt_j) = j$.
	For every $2 \leq i \leq k-1$, there exists
	\[ g_i \in \langle \bt_1, \dots, \bt_{i-1} \rangle, \quad \deg_\bt(g_i) = i-1, \]
	such that, if we symbolically compute
	\begin{equation}
		\label{eq:zuppaefagioli}
		S = \big( X:1:\ff(X) \big) + \left( \sum_{j=1}^k \bt_j X^j :1:\ff\left( \sum_{j=1}^k \bt_j X^j\right) \right),
	\end{equation}
	then the coefficient of $X^i$ in $S_x$ is $\bt_i + g_i$.
\end{lemma}
\begin{proof}
	Let us denote for simplicity the two points involved in the sum \eqref{eq:zuppaefagioli} by $P$ and $Q$, respectively. 
	We observe that the coefficient of $X^j$ in $Q_x$ has always $\deg_\bt$ equal to $j$. 
	Since $Q_z = \ff(Q_x)$, this also holds for $Q_z$.
	On the other hand, the coefficients of every power of $X$ have $\deg_\bt$ equal to $0$ in both $P_x$ and $P_z$.
	Since the addition formulae are polynomials in the entries, if we write
	\[ S_x = \sum_{i = 1}^k \Psi_i X^i, \]
	then we have $\deg_\bt(\Psi_i) \leq i$.
	We now look at all the terms in $\Psi_i$ with $\deg_\bt$ at least $i-1$, namely those terms that involve at most one time $P_x$, and that never involve $P_z$.
	A close inspection of the addition law (detailed in the appendix, see Proposition \ref{prop:sommemodulop2} with $P_1=Q$ and $P_2=P$) shows that these terms only arise from
	\begin{equation} \label{eq:contacci}
		P_x + Q_x + \big(a_1Q_x - a_2 Q_x^2 + 2a_3 Q_z - 2a_4 Q_xQ_z - 3a_6 Q_z^2\big)P_x.
	\end{equation}
	Since we are considering $i \geq 2$, then $P_x = X$ alone does not produce any term in $\Psi_i$.
	Instead, the term $\bt_i$ of $Q_x$ appears in $\Psi_i$, and it is therefore its unique term of maximal $\deg_\bt$.
	We now show that the element
	\[ g_i = \Psi_i - \bt_i \]
	is the required polynomial.
	By construction we have $\deg_\bt(g_i) \leq i-1$, but from equation \eqref{eq:contacci} we see that it always contain the term $a_1 \bt_{i-1}$, which has $\deg_\bt$ equal to $i-1$. Hence, we have $\deg_\bt(g_i) = i-1$.
	Finally, an easy inspection of the formulae of Proposition \ref{prop:shortsum} shows that the unique term of $P$ that never appears in $S_x$ multiplied by any term of $Q$ is $P_x = X$.
	However, this term only appears in $\Psi_1$, therefore for all $i \geq 2$ every monomial composing $\Psi_i$ is divisible by some $\bt_j$.
	Furthermore, we have $\deg_\bt(\bt_j) = j$, so $\bt_j$ cannot appear in $g_i$ for every $j \geq i$. In conclusion, we have
	\[ g_i \in \langle \bt_1, \dots, \bt_{i-1} \rangle, \]
	so all such $g_i$'s have the required properties.
\end{proof}
We are now ready to prove the main results of this section.
\begin{theorem}
	\label{teo:psi_poli}
	For every $1 \leq i \leq k-1$, the $i$-th multiplication polynomial $\psi_i$ is a polynomial in $\Q[a_1, \dots, a_6][n]$ of degree $i$ in $n$.
	Moreover, we have $n|\psi_i(n)$.
\end{theorem}
\begin{proof}
	The case $i=1$ follows from Lemma \ref{lem:psi1}.
	We prove the thesis by extended induction on $i\geq 2$.
	The base case $i=2$ is given by Lemma \ref{lem:psi2}.
	We now assume that this result holds for every $j \leq i-1$, and we show that this implies it also holds for $i$.
	Since the addition law is associative, we have
	\[ (nP)_x = \big(P + (n-1)P\big)_x. \]
	The coefficient of $X^i$ on the left-hand side is $\psi_i(n)$. The right-hand side is given by Equation \eqref{eq:zuppaefagioli} after substituting $\bt_j = \psi_j(n-1)$.
	Hence by applying Lemma \ref{lem:contirognosi} we obtain
	\[ \psi_i(n) - \psi_i(n-1) = g_i\big(\psi_1(n-1), \dots, \psi_{i-1}(n-1)\big). \]
	By the inductive hypothesis, $\psi_j(n-1)$ is a degree-$j$ polynomial in $n$ without the constant term.
	Since $\deg_\bt(\bt_j) = j$, $\deg_\bt(g_i) = i-1$ and $g_i$ has no constant terms, then the evaluation of $g_i$ by $\bt_j = \psi_j(n-1)$ produces a degree-$(i-1)$ polynomial in $n-1$ without a constant term, namely
	\[ g_i\big(\psi_1(n-1), \dots, \psi_{i-1}(n-1)\big) = c_1 (n-1) + \dots + c_{i-1} (n-1)^{i-1}, \]
	for some coefficients $c_j \in \Q[a_1, \dots, a_6]$, and $c_{i-1} \neq 0$.
	
	The above arguments hold uniformly on $n$, therefore we have a system of relations
	\[ \begin{cases}
		\psi_i(n) - \psi_i(n-1) &= c_1 (n-1) + \dots + c_{i-1} (n-1)^{i-1}, \\
		\psi_i(n-1) - \psi_i(n-2) &= c_1 (n-2) + \dots + c_{i-1} (n-2)^{i-1}, \\
		&\vdots \\
		\psi_i(2) - \psi_i(1) &= c_1 + \dots + c_{i-1}.
	\end{cases}\]
	We recall that for every $i \geq 2$, we have $\psi_i(1) = 0$ by Remark \ref{rmk:psii1}. 
	Therefore, by adding all the above relations we obtain
	\begin{equation}
		\label{eq:psi_i_fab} \psi_i(n) = c_1\Sum_{m=1}^{n-1} m + c_2\Sum_{m=1}^{n-1}m^2 + \cdots + c_{i-1} \Sum_{m=1}^{n-1}m^{i-1}.
	\end{equation}
	Thanks to Faulhaber formulas \cite[Sec. 6.5]{graham}, the sum
	\[ S_j(n) = 1^j + 2^j + \cdots + (n-1)^j = \Sum_{m=1}^{n-1} m^j \]
	can be expressed in a closed form as a polynomial in $\Q[n]$ of degree $j+1$ without a constant term.
	Since
	\begin{equation}
		\label{eq:psi_closedsum}
		\psi_i(n) = \Sum_{j=1}^{i-1} c_j\Sum_{m=1}^{n-1} m^j = \Sum_{j=1}^{i-1} c_j S_j(n),
	\end{equation}
	then also $\psi_i(n)$ can be expressed as a polynomial in $\Q[a_1,\dots,a_6][n]$ of degree $i$ and with constant term equal to $0$, namely $n|\psi_i(n)$.
\end{proof}

The explicit polynomials $\psi_i(n)$ for the first values of $i$, as well as further details on their computation, can be found in the Appendix (Section \ref{sec:coeff}).
A more extensive list can be found at \cite{repo}.

\begin{remark}
	Whenever the elliptic curve is fixed, the $a_i$'s are fixed elements of $\rcal$, therefore Theorem \ref{teo:psi_poli} implies that $\psi_i(n)$ will have degree at most $i$ in $n$.
	However, its degree might well be strictly lower than $i$, for instance when dealing with short Weierstrass forms $(a_1 = a_2 = a_3 = 0)$.
\end{remark}
\begin{remark}
	While being a polynomial in $\Q[a1,\dots,a_6][n]$, when evaluated in a specific $\bar{n}\in\N^*$ we always get $\psi_i(\bar{n}) \in \Z[a_1,\dots,a_6]$. This is expected since we are dealing with integer quantities when adding points.
	For $\psi_1(\bar{n})=\bar{n}$ is clear, and also $\psi_2(\bar{n}) \in \Z[a_1,\dots,a_6]$, since one among $\bar{n}$ and $\bar{n}-1$ will be even.
	Following the same induction performed in Theorem \ref{teo:psi_poli}, we eventually conclude that $\psi_i(\bar{n}) \in \Z[a_1,\dots,a_6]$ for every $1 \leq i \leq k-1$.
\end{remark}

\begin{notation*} We denote the product of the first $i \in \N^*$ factorials as
	\[ \Pi(i) = \prod_{j=1}^i j!. \]
\end{notation*}
\begin{lemma}
	\label{lem:multinomial}
	Let $i, j_1, \dots, j_m \in \N^*$ such that $\sum_{l=1}^m j_l \leq i$. Then
	\[ \Pi(j_1) \cdots \Pi(j_m) | \Pi(i). \]
\end{lemma}
\begin{proof}
	We prove this by induction on $i$. For $i=1$ there is nothing to prove. Now let us assume that it holds for $i-1$ and every possible $j_1,\dots,j_m \in \N^*$ such that $\sum_{l=1}^m j_l \leq i-1$. Let $j_1,\dots j_m\in \N^*$ be such that $\sum_{l=1}^m j_l \leq i$.
    Notice that in general $\Pi(i) = i!\Pi(i-1)$.
    By the Multinomial Theorem, we have $j_1!\cdots j_m! | i!$, and by inductive hypothesis $\Pi(j_1-1) \cdots \Pi(j_m-1) | \Pi(i-1)$, hence the thesis follows by multiplying the previous relations.
\end{proof}
\begin{theorem}
	\label{teo:psi_int}
	With the above notation, for every $1 \leq i \leq k-1$ we have
	\[ \Pi(i) \psi_i(n) \in \Z[a_1,...,a_6][n]. \]
\end{theorem}
\begin{proof}
	By Theorem \ref{teo:psi_poli} we know that $\psi_i(n) \in \Q[a_1,...,a_6][n]$.
	The case $i=1$ follows from Lemma \ref{lem:psi1}.
	We prove the thesis by extended induction on $i\geq 2$, where the base case $i=2$ is given by Lemma \ref{lem:psi2}.
	Let us assume that the thesis holds for every $j < i$. With the notation of Equation \eqref{eq:psi_closedsum} we prove that
	\begin{enumerate}[label=(\roman*)]
		\item \label{item:1} for every $1 \leq j \leq i-1$ we have $\Pi(i-1) c_j \in \Z[a_1,...,a_6]$,
		\item \label{item:2} for every $1 \leq j \leq i-1$ we have $i! S_j(n) \in \Z[a_1,...,a_6]$.
	\end{enumerate}
	By combining \ref{item:1} and \ref{item:2} with Equation \eqref{eq:psi_closedsum}, the thesis follows.
	
	\ref{item:1}: By Lemma \ref{lem:contirognosi} we know that $\deg_\bt(g_i) = i-1$ and that $g_i$ arises as a polynomial in the $\psi_j(n-1)$ and coefficients in $\Z[a_1,\dots,a_6]$, whose monomials $\psi_{j_1}(n-1) \cdots \psi_{j_m}(n-1)$ satisfy $\sum_{l=1}^m j_l \leq i-1$.
	By the inductive hypothesis, for all such monomials we have
	\[ \Pi(j_1) \cdots \Pi(j_m) \psi_{j_1} \cdots \psi_{j_m} \in \Z[a_1,...,a_6][n]. \]
	Since $\Pi(j_1) \cdots \Pi(j_m)$ always divides $\Pi(i-1)$ by Lemma \ref{lem:multinomial}, then
	\[ \Pi(i-1) g_i\big(\psi_1(n-1), \dots, \psi_{i-1}(n-1)\big) \in \Z[a_1, \dots,a_6][n]. \]
	\ref{item:2}: From \cite[Eq. 6.80]{graham} we obtain
	\[ (m+1)S_m(n) = n^{m+1} - \Sum_{j=0}^{m-1} \binom{m+1}{j}S_j(n). \] 
	Since $2S_1(n) = n^2 - n$, a simple induction on $m \geq 1$ shows that we have $(j+1)!S_j(n) \in \Z[n]$, therefore also $i!S_j(n) \in \Z[n]$.
\end{proof}
\begin{cor}
	\label{cor:psi_i_zero}
	Let $p$ be a prime number. For every exponent $l \geq 1$ and for every $1 \leq i < p$, we have
	$$\psi_i(p^l) \equiv 0 \ \bmod p^l.$$
\end{cor}
\begin{proof}
	By Theorem \ref{teo:psi_int} we have $\Pi(i) \psi_i(p^l) \in \Z[a_1,\dots,a_6]$.
	Thanks to Theorem \ref{teo:psi_poli} we also have $\Pi(i) \psi_i(p^l) \equiv 0 \mod p^l$.
	By definition $\Pi(i) \in (\Z/p^l\Z)^*$ for every $i < p$, therefore we conclude $\psi_i(p^l)\equiv 0 \mod p^l$.
\end{proof}

\begin{remark}
	As $\rcal$ is a finite local ring, its characteristic $\chr(\rcal)$ is a prime power $p^l$.
	Thus, Corollary \ref{cor:psi_i_zero} implies that
	for every $P=\big(X:1:\ff(X)\big) \in E^{\infty}$ we have $\psi_i(p^l) = 0$ for every $i<p$, hence
	\[ X^p \ | \ (p^lP)_x. \]
	Furthermore, since $X \in \m$, this implies $(p^lP)_x \in \m^p$.
	This condition imposes severe restrictions on the possible group structures arising from $E^\infty$.
\end{remark}
\section{Elliptic curves over \texorpdfstring{$\F_q[x]/(x^k)$}{Fq[x]/(x\^k)}} \label{sec:GrpStructure}
In this section, we fix a prime power $q = p^e$ and a positive integer $k \in \N^*$.
Let $\F_q$ be the finite field of size $q$, and consider the ring
\[ R_k = \F_q[x]/(x^k) \simeq \F_q[\eps], \ \textnormal{with } \eps^k = 0. \]
Such an $R_k$ is a finite local ring, whose maximal ideal $\m = (\eps)$ is principal, therefore it underlies the results of the previous sections.
From now on, we will work over the ring $\rcal = R_k$.
Notice that $k$ is the nilpotence degree of $R_k$, consistently with our previous notation.
If $k = 1$ then $R_k \simeq \F_q$, while if $k \geq 2$ then every element $r \in R_k$ may be written uniquely as $r = a + b\eps$, for some $a \in \F_q$ and a degree-$(k-2)$ polynomial $b \in \F_q[\eps]$.
Moreover, $r \in R_k^*$ if and only if $a \neq 0$. Hence, every $r \in R_k$ is either a unit or divisible by $\eps$.

We will exploit the multiplication polynomials to compute the group structure of elliptic curves over this ring in all but a few exceptional cases given by particular choices of the curve coefficients $a_1, \dots, a_6 \in R_k$. 

The canonical projection may be written explicitly as
\begin{equation*}
    \begin{gathered}
        \pi : E(R_k) \rightarrow E(\F_q),\\
        (\alpha_x + \beta_x \eps:\alpha_y + \beta_y \eps:\alpha_z + \beta_z \eps) \mapsto (\alpha_x:\alpha_y:\alpha_z).
    \end{gathered}
\end{equation*}
From \cite[Sec. 4]{lenstra86} we know that its fibers have size $q^{k-1}$, so in particular $E^\infty = \pi^{-1}(\ocal)$ is a $p$-subgroup of $E(R_k)$.
Moreover, the structure of $E^\infty$ often prescribe the structure of the whole group, as we have the short exact sequence of groups
\[ 0 \rightarrow E^\infty(R_k) \xhookrightarrow{i} E(R_k) \overset{\pi}{\twoheadrightarrow} E(\F_q) \rightarrow 0. \]
When the above sequence splits, we have
\[ E(R_k) \cong E(\F_q) \oplus E^\infty. \]
This is always the case when
\begin{equation} \label{eq:CoprimalityCondition}
    \gcd\big(\#E(\F_q), p\big) = 1,
\end{equation}
which happens with overwhelming probability for large primes $p$, and it is always satisfied by elliptic curves of cryptographic interest.
We will therefore address the group structure of elliptic curves underlying this condition.

We can now apply Corollary \ref{cor:psi_i_zero} to this setting (with $l=1$).
\begin{cor}
    \label{cor:pmul_jump}
    Let $P = \big(X:1:\ff(X)\big) \in E^\infty$ be a point.
    Then
    \[ (pP)_x \equiv \psi_p(p) X^p \bmod X^{p+1}. \]
\end{cor}
This result is sufficient to compute the group structure of $E^\infty$ whenever considered exponent $k$ is smaller than the ring characteristic $p$.
\begin{prop}
    \label{prop:casofacile}
    Let $E$ be an elliptic curve over $R_k$ where $k \leq p$. Then we have the group isomorphism
    $$E^\infty \cong (\F_p)^{e(k-1)}.$$
\end{prop}
\begin{proof}
    From Corollary \ref{cor:pmul_jump} every point has order $p$. 
    We already observed that $E^\infty$ is a $p$-group of size $q^{k-1} = p^{e(k-1)}$, from which the thesis follows.
\end{proof}
To address the cases with $k > p$ we by introduce a way for "counting the divisibility" of points with respect to $\eps$. 
\begin{defi}
    Let $r \in R_k \backslash \{0\}$.
    We define its \emph{minimal degree} $\vi(r)$ as the maximal $i \geq 0$ such that $\eps^i | r$.
    We also define $\vi(0) = \infty$.
    Finally, for every point $P \in E^\infty$, we define $\vi(P) = \vi(P_x)$.
\end{defi}
We notice that $\vi$ is almost a valuation on $R_k$, as it satisfies
\[ \vi(xy) \geq \vi(x)\vi(y), \quad \text{and} \quad \vi(x+y) \geq \min\{ \vi(x), \vi(y) \}. \]
\begin{remark}
    $\vi(P) = \infty$ if and only if $P = \ocal$.
\end{remark}
\begin{lemma}
    \label{lem:degsum}
    Let $P,Q \in E^\infty$ be two points with $\vi(P) \neq \vi(Q)$. Then
    $$\vi(P+Q) = \min\{\vi(P), \vi(Q)\}.$$
\end{lemma}
\begin{proof}
    The statement is trivial if either $P$ or $Q$ is $\ocal$, so let us assume $P,Q \neq \ocal$.
    Let us denote $m = \min\{\vi(P), \vi(Q)\}$.
    Then $e^{m+1}$ divides $P_z$, $Q_z$ and all the products involving $x$ and $z$ coordinates of both the points. A close inspection of the formulae (detailed in Proposition \ref{prop:inspection_eps}, with $P_1=P$ and $P_2=Q$) shows that
    \[
    (P+Q)_x \equiv P_x + Q_x \bmod \eps^{m+1}
    \]
    Since by assumption $\vi(P_x) \neq \vi(Q_x)$, the conclusion follows.
\end{proof}
\begin{lemma}
    \label{lem:samedegsum}
    Let $P,Q \in E^\infty$ be points with $\vi(P) = \vi(Q) = m < \infty$.
    Let $c_p, c_q \in \F_q$ be the coefficients such that
    \[ P_x \equiv c_p\eps^m \bmod \eps^{m+1} \quad \textnormal{and} \quad Q_x \equiv c_q\eps^m \ \bmod \eps^{m+1}. \]
    Then we have
    $$(P+Q)_x = (c_p + c_q)\eps^m \bmod \eps^{m+1}.$$
\end{lemma}
\begin{proof}
    It follows from the same argument of Lemma \ref{lem:degsum}.
\end{proof}
\begin{remark}
    In the notation of Lemma \ref{lem:samedegsum}, if $c_p + c_q \neq 0$ then 
    \[ \vi(P+Q) = m = \vi(P) = \vi(Q). \]
\end{remark}
\begin{lemma}
    \label{lem:deg_nonpmul}
    Let $n \in \N^*$ such that $p \nmid n$. Then
    $$\vi(nP) = \vi(P).$$ 
\end{lemma}
\begin{proof}
    Thanks to Lemma \ref{lem:psi1}, we have
    \[ (nP)_x \equiv n P_x \bmod P_x^2. \]
    Since $n$ is invertible in $\F_q$, we have $\vi(nP_x) = \vi(P_x)$.
    If $P = \ocal$ the thesis is clear.
    Otherwise, we have $\vi(nP_x) < \vi(P_x^2)$, therefore by Lemma \ref{lem:degsum} we have
    \[ \vi\big((nP)_x\big) = \vi(nP_x) = \vi(P_x). \]
    The quantity on the left side is $\vi(nP)$, while the right one is $\vi(P)$, so they are equal.
\end{proof}
\begin{defi}
    We define the \emph{trajectory} of $P \in E^\infty$ as
    \[ \trj(P) = \{ \vi(nP) \}_{n \in \N} \setminus \{ \infty \}. \] 
\end{defi}
\begin{example}
    If a point $P$ has order $p$ and $\vi(P) = m$, by applying Lemma \ref{lem:deg_nonpmul} we see that $\trj(P) = \{m\}$.
\end{example}
\begin{lemma}
    \label{lem:disj_trj}
    Let $P,Q\in E^\infty$ be points with $\trj(P) \cap \trj(Q) = \emptyset$. For every $n,m \in \Z$ we have
    $nP + mQ = \ocal$ if and only if $nP = mQ = \ocal$.
\end{lemma}
\begin{proof}
    If both $nP$ and $mQ$ are $\ocal$, then also their sum clearly is.
    On the other side, since $\trj(P) \cap \trj(Q) = \emptyset$ we have $\vi(nP) \neq \vi(mQ)$, which by Lemma \ref{lem:degsum} implies
    \[ \infty = \vi(nP + mQ) = \min\{ \vi(nP), \vi(mQ) \}. \]
    Therefore we have $\vi(nP) = \vi(mQ) = \infty$, i.e. $nP = mQ = \ocal$.
\end{proof}

As we will see shortly, the group structure of $E^\infty$ depends on $\vi\big(\psi_p(p)\big)$.
There are two possible cases:
\begin{enumerate}[ label=$(P_{\arabic*})$ ]
    \item \label{case:main} $\vi\big(\psi_p(p)\big) = 0$, i.e. $\psi_p(p) \in R_k^*$, or
    \item \label{case:exceptional} $\vi\big(\psi_p(p)\big) > 0$, i.e. $\eps | \psi_p(p)$.
\end{enumerate}
The first case will be referred to as the \emph{main case}, as it occurs with overwhelming probability with a uniform choice of the curve coefficients, and it will be discussed in Section \ref{subs:main_case}. 
The second case will be referred to as the \emph{exceptional case}, and examined in Section \ref{subs:except_case}.

\subsection{Main case} \label{subs:main_case}
In this section, we focus on case \ref{case:main}, namely we will assume that 
\begin{equation} \label{eq:MainCondition}
    \psi_p(p) \in R_k^*.
\end{equation} 
The main idea is to 
use Corollary \ref{cor:pmul_jump} to partition the numbers up to $k-1$ in different point trajectories.
For each trajectory, we will pick $e$ independent points $P_i$ such that $\vi(P_i)$ is the minimal value of the trajectory and show that these points generate the whole group.
\begin{lemma}
    \label{lem:deg_pmul1}
    Let $P \in E^\infty$ be a point. For every $i \in \N$, we have
    $$\vi(p^i P) = 
    \begin{cases}
    p^i\vi(P) &\text{ if } \ p^i\vi(P) < k, \\
    \infty &\text{ otherwise.}
    \end{cases}
    $$
\end{lemma}
\begin{proof}
    We prove it by induction on $i$.
    The base step $i=0$ holds identically.
    Let us now assume this holds for $i = j-1$. Then
    \[
    \vi(p^jP) = \vi\big(p(p^{j-1}P)\big) = \vi(pQ), \quad \textnormal{where} \quad Q = p^{j-1}P.
    \]
    By inductive hypothesis $\vi(Q) = p^{j-1}m$ so we can write
    \[
    Q_x = c_m \eps^{p^{j-1} m} \bmod \eps^{p^{j-1}m + 1}
    \]
    for some $c_m \in R_k^*$. Then by Corollary \ref{cor:pmul_jump} we obtain
    \[ 
    (pQ)_x = \psi_p(p)(c_m)^p \eps^{p^jm} \bmod \eps^{p^jm + 1}.
    \]
    By assumption \eqref{eq:MainCondition} we know that $\psi_p(p) \in R_k^*$, which implies that $\vi(pQ) = p^jm = p\vi(Q)$ if $pm < k$, and $\vi(pQ) = \infty$ otherwise.
\end{proof}
\begin{prop}
    \label{prop:ordp}
    For every $1 \leq m \leq k-1$, if $P \in E^\infty$ has minimal degree $m = \vi(P)$, then its order is
    \[ \ord(P) = p^{l_m}, \quad \textnormal{where} \quad l_m = \left\lfloor\log_p  \frac{k-1}{m}\right\rfloor + 1. \]
\end{prop}
\begin{proof}
    By definition, the integer $l_m$ is the largest integer such that $mp^{l_m-1} \leq k-1$, while $mp^{l_m} > k-1$. In fact, we have
    \[
    mp^{l_m-1} = mp^{\left\lfloor \log_p \frac{k-1}{m}\right\rfloor} \leq m\frac{k-1}{m} = k-1,
    \]
    and
    \[
    mp^{l_m} = mp^{\left\lfloor \log_p \frac{k-1}{m}\right\rfloor + 1} > m\frac{k-1}{m} = k-1.
    \]
    From Lemma \ref{lem:deg_pmul1} we see that $\vi(p^{l_m - 1}P) = mp^{l_m - 1} \leq k-1$, and then $p^{l_m - 1}P \neq \ocal$, while $\vi(p^{l_m}P) = mp^{l_m} > k-1$, hence $p^{l_m} = \ocal$. The thesis follows from the fact that $E^\infty$ is a $p$-group.
\end{proof}
\begin{remark}
    For every $1 \leq m \leq k-1$, the quantity $l_m$ given in Proposition \ref{prop:ordp} is well defined, since the point $P = \big(\eps^m:1:\ff(\eps^m)\big)$ satisfies $\vi(P) = m$.
\end{remark}
\begin{lemma}
    \label{lem:trj_case1}
    For every $1 \leq m \leq k-1$, if $P \in E^\infty$ has minimal degree $m = \vi(P)$, then
    \[ \trj(P) = \{mp^i\}_{i < l_m}. \]
\end{lemma}
\begin{proof}
    It follows from Lemma \ref{lem:deg_pmul1} and Proposition \ref{prop:ordp}.
\end{proof}
\begin{lemma}
    \label{lem:trj_size}
    For every $1 \leq m \leq k-1$, if $P \in E^\infty$ has minimal degree $m = \vi(P)$, then $\#\trj(P) = l_m$.
\end{lemma}
\begin{proof}
    With the same notation of Lemma \ref{lem:trj_case1}, the $\{mp^i\}_{i < l_m}$ are all distinct integers and the index $i$ runs from $0$ to $l_m - 1$.
\end{proof}
\begin{lemma}
    \label{lem:sum_li}
    With the above notation, we have
    $$ 
    \Sum_{\substack{1 \leq m \leq k-1\\(m,p)=1 }} l_m = k-1.
    $$
\end{lemma}
\begin{proof}
    By Lemma \ref{lem:trj_case1} and \ref{lem:trj_size}, we know that $l_m$ is the size of $\{mp^i\}_{i < l_m}$.
    It is enough to show that these sets for $(m,p)=1$ form a partition of the numbers between $1$ and $k-1$.
    They are disjoint, since $m_1p^{h_1} = m_2p^{h_2}$ with $m_1, m_2$ coprime with $p$ is possible only if $m_1=m_2$.
    Moreover, every number below $k-1$ can be written as $mp^h$ for some $m \leq k-1$ coprime with $p$ and $h < l_m$, by definition of $l_m$. This completes the proof.
\end{proof}
\begin{prop}
    \label{prop:indipendenza}
    Let $\{\gm_n\}_{1 \leq n \leq e}$ be an $\F_p$-basis of $\F_q$. For any given $1 \leq m \leq k-1$, the points
    \[ g_{nm} = \big(\gm_n\eps^m:1:\ff(\gm_n\eps^m)\big) \in E^\infty \]
    are linearly independent.
    Moreover, the trajectory of every linear combination of the $\{g_{nm}\}_{1 \leq n \leq e}$ lies into $\{ m p^j \}_{j < l_m}$.
\end{prop}
\begin{proof}
    We want to show that for every $h_{nm} \in \N$ we have
    \[ S = \Sum_{n=1}^e h_{nm} g_{nm} = \ocal \iff \forall \ 1 \leq n \leq e: \ h_{nm}g_{nm} = \ocal. \]
    Clearly $h_{nm}g_{nm} = \ocal$ for every $n$ implies $S = \ocal$.
    On the other side, let us suppose that $h_{nm}g_{nm} \neq \ocal$ for some $n$, i.e. $\vi(h_{nm}g_{nm}) < \infty$.
    Let $\mu < \infty$ be the minimal degree of such points, we will show that also $\vi(S) = \mu$, hence $S \neq \ocal$.
    Let $N$ be the set of $n$'s achieving this minimum, i.e. $\vi(h_{nm}g_{nm}) = \mu$ if $n \in N$, and $\vi(h_{nm}g_{nm}) > \mu$ otherwise.
    Since $\vi(g_{nm}) = m$ by construction, then by Lemma \ref{lem:deg_pmul1} there is $i \in \N$ such that for every $n \in N$ there exists $h_n \in \N$ with $(h_n,p)=1$, and satisfying $h_{nm} = h_np^i$, and $mp^i = \mu$.
    By Lemma \ref{lem:samedegsum} this implies
    \[ S_x = \psi_p(p)^i \left(\Sum_{n\in N} h_n \gamma_n\right) \eps^{\mu} \bmod \eps^{\mu + 1}. \]
    Since $\psi_p(p) \in R_k^*$ by assumption \eqref{eq:MainCondition} and $\Sum_{n\in N} h_n \gamma_n$ is a non-zero element of $\F_q$, then we conclude that $\vi(S_x) = \mu < \infty$, so $S \neq \ocal$.
    Moreover, by Lemma \ref{lem:trj_case1} we have $\trj(S) = \{ \mu p^{i}\}_{i<l_{\mu}}$, which is contained in $\{ m p^j \}_{j < l_m}$ as $\mu = mp^i$.
\end{proof}
\begin{teo}
    \label{teo:grp_struct_c1}
    Let $E$ be an elliptic curve over $R_k$ satisfying the condition \eqref{eq:MainCondition}. Then
    $$ E^\infty \cong \Prod_{\substack{1 \leq m \leq k-1\\(m,p)=1 }} \left( \Z_{ p^{l_m} } \right)^e.$$
\end{teo}
\begin{proof}
    Let $\{\gm_n\}_{1 \leq n \leq e}$ be an $\F_p$-basis of $\F_q$.
    For every $1 \leq n \leq e$ and $1 \leq m \leq k-1$ such that $(m,p)=1$, we define
    \[ g_{nm} = \big(\gm_n\eps^m:1:\ff(\gm_n\eps^m)\big) \in E^\infty, \]
    as in Proposition \ref{prop:indipendenza}. 
    We will show that these points are linearly independent and generate the whole $E^\infty$. 

    Let us assume that there are $\{h_{nm}\}_{n,m} \subset \N$ such that
    $$
    \Sum_{(m,p) = 1}^{k-1} \Sum_{n=1}^e h_{nm} g_{nm} = \ocal.
    $$
    By Proposition \ref{prop:indipendenza} we have
    \[ \trj\left(\Sum_{n=1}^e h_{nm} g_{nm}\right) \subseteq \{mp^j\}_{j < l_m}, \]
    which are all disjoint for $(m,p) = 1$.
    Thus, by repeatedly applying Lemma \ref{lem:disj_trj}, we conclude that $\Sum_{n=1}^e h_{nm} g_{nm} = \ocal$ for all the considered $m$.
    But for every fixed $m$, the point $g_{nm}$ are linearly independent by Proposition \ref{prop:indipendenza}, therefore we conclude that $g_{nm} = \ocal$, for every considered $n$ and $m$.
    
    By Lemma \ref{prop:ordp} we have $\ord(g_{nm}) = p^{l_m}$, therefore
    \[ \Prod_{\substack{1 \leq m \leq k-1\\(m,p)=1 }} \left( \Z_{ p^{l_m} } \right)^e \cong \langle g_{nm} \rangle_{n,m} \subseteq E^{\infty}. \]
    The conclusion follows by comparing the sizes, since $\#\Z_{ p^{l_m} }^e = q^{l_m}$ and by Lemma \ref{lem:sum_li} we have
    \[ \prod_{\substack{1 \leq m \leq k-1\\(m,p)=1 }} q^{l_m} = q^{k-1}, \]
    which is precisely $\#E^{\infty}$.
\end{proof}

\begin{cor}
    \label{cor:grp_struct}
    Let $E$ be an elliptic curve over $R_k$ satisfying the conditions \eqref{eq:CoprimalityCondition} and \eqref{eq:MainCondition}. Then
    $$ E \cong E(\F_q) \times \Prod_{\substack{1 \leq m \leq k-1\\(m,p)=1 }} \left( \Z_{ p^{l_m} } \right)^e, \quad \textnormal{where} \quad l_m = \left\lfloor\log_p  \frac{k-1}{m}\right\rfloor + 1. $$
\end{cor}
\begin{proof}
    Under condition \eqref{eq:CoprimalityCondition} we know that
    \[ E(R_k) \cong E(\F_q) \oplus E^\infty, \]
    while under condition \eqref{eq:MainCondition} the group structure of $E^\infty$ is given by Theorem \ref{teo:grp_struct_c1}.
\end{proof}
Since both the conditions \eqref{eq:CoprimalityCondition} and \eqref{eq:MainCondition} are satisfied with overwhelming probability for large primes $p$, the group structure of a generic elliptic curve over $R_k$ is the one given by Corollary \ref{cor:grp_struct}.

\subsection{Exceptional case} \label{subs:except_case}
For some special choices of the curve coefficients $a_1, \dots, a_6 \in R_k$, the condition \eqref{eq:MainCondition} may not hold.
In this final section, we examine these cases, namely we will always assume that
\[ \eps \ | \ \psi_p(p). \]
This heuristically happens with a probability slightly higher than $1/p$ (Section \ref{sec:exeption_prob}).
Moreover, if we fix the $a_i$ in $F_q$ (for example when lifting a curve in $F_q$) it cannot happen that $\eps | \psi_p(p) $, but only $\psi_p(p) = 0$.
In this case, Proposition \ref{prop:casofacile} can be slightly extended, as follows.
\begin{prop}
    \label{prop:casopiufacile}
     Let $E$ be an elliptic curve over $R_k$ with $k \leq p+1$. Then we have the group isomorphism
    $$E^\infty \cong (\F_p)^{ek}.$$
\end{prop}
\begin{proof}
    If follows as in Proposition \ref{prop:casofacile}, since every point of $E^\infty$ has order $p$ by Corollary \ref{cor:pmul_jump}.
\end{proof}
The following example shows why Theorem \ref{teo:grp_struct_c1} does not hold in this case: depending on the value of $\psi_p(p)$, the points may have different trajectories.
However, as long as they remain disjoint, the group structure can still be computed.
\begin{example}
    \label{ex:strange_ex1}
    Let us consider $q=p=3$ and $k=20$, namely the ring $\rcal=\F_3[x]/(x^{20}) \simeq \F_3[\eps]$, and the curve $E$ defined over $\rcal$ by
    \[ y^2z + \eps^4 xyz = x^3 + \eps^8 x^2z + xz^2. \]
    One can check that $\Dt_E \in \rcal^*$, hence $E$ is an elliptic curve.
    We obtain $\psi_3(3) = 2\eps^8$ and $\psi_9(3) = 2 + \eps^{16}$, while all the other $\psi_i(3)$ for $i<9$ are equal to $0$.
    Since $\vi\big(\psi_3(3)\big) = 8$, this is the exceptional case, so Theorem \ref{teo:grp_struct_c1} does not hold anymore, but we will see that a slight modification of it does.
    For $1\leq m \leq 19$, the trajectories of the points $P_m = \big(\eps^m:1:\ff(\eps^m)\big)$ are
    \[ \trj(P_1) = \{ 1, 9 \}, \quad \trj(P_2) = \{ 2, 14 \}, \quad \trj(P_3) = \{ 3, 17\}, \]
    while all the other $P_m$ have order $3$ (hence $\trj(P_m) = \{m\}$). In fact, the triple of a generic $P = \big(X:1:\ff(X)\big) \in E^{\infty}$ satisfies
    \[ (3P)_x \equiv \psi_3(3)X^3 + \psi_9(3)X^9 \bmod X^{10}. \]
    If we call $m = \vi(P)$, then there are only two possibilities for $\vi(3P)$:
    \begin{itemize}
        \item if $3m + 8 < 9m$, the power of $\eps$ with minimal degree arises from $\psi_3(3)X^3$, hence $\vi(3P) = 3m + 8$;
        \item otherwise, the minimal degree arises from $X^9$, as $\psi_9(3) \in \rcal^*$, therefore in this case $\vi(3P) = 9m$.
    \end{itemize}
    Thus, we define the set $\acal = \{ 1 \dots 19 \} \backslash \{9, 14, 17\}$, and the integers $l_1 = l_2 = l_3 = 2$, while $l_m = 1$ for all the other $m \in \acal$.
    The trajectories of the $\{P_m\}_{m \in \acal}$ are all disjoint, hence we can follow the proof of Proposition \ref{prop:indipendenza} to show that the $P_i$ are linearly independent. The group they generate is
    \begin{equation*}
        G = \Prod_{m \in \acal} \Z_{p^{l_m}}.
    \end{equation*} 
    As in the proof of Lemma \ref{lem:sum_li}, we have $\sum_{m \in \acal} l_m = 19 = k-1$, since the trajectories of the $P_m$ for $m \in \acal$ partition the set $\{1,...,19\}$ by construction.
    This implies that $G$ has exactly $p^{k-1}$ elements, hence it is the whole $E^\infty$.
    More details on the explicit computations can be found in the Appendix (Example \ref{ex:strange_ex1_calc}).
\end{example}

The main idea of Theorem \ref{teo:grp_struct_c1} is finding points with non-intersecting trajectories, which will generate $E^\infty$.
Example \ref{ex:strange_ex1} shows that this idea may also be adapted to the exceptional case. We recall that in this case it holds $\vi\big(\psi_p(p)\big)=d>0$.
Moreover, we assume that the following three conditions hold.
\begin{enumerate}[ label=$(C_{\arabic*})$ ]
    \item\label{c1} $\psi_{p^2}(p) \in R_k^*$,
    \item\label{c2} $\psi_i(p) = 0$ for all $i < p^2$ such that $(i,p)=1$,
    \item\label{c3} $\psi_i(p) \in \langle \psi_p(p) \rangle$ for all $i<p^2$ such that $p | i$.
\end{enumerate}
We will show that under these conditions we can compute the group structure of $E^\infty$ as we did in the main case. After that, we verify that these conditions always hold within the reach of our computations.
\begin{lemma}
    \label{lem:except_jump}
    Let $P \in E^\infty$ be a point. We have
    \[ \vi(pP) = \min \{p^2\vi(p), p\vi(p)+d\}. \]
\end{lemma}
\begin{proof}
    All the $\psi_i(p)$ with $i<p^2$ are either $0$ if $(i,p)=1$ for \ref{c2}, or multiple of $\psi_p(p)$ if $p|i$ for \ref{c3}.
    Moreover, every $i > p^2$ cannot be the minimal degree of $pP$ since $\psi_{p^2}(p) \in R_k^*$ for \ref{c1}.
    Therefore, $\vi(pP)$ is always determined by the minimal degree of $\psi_p(p)X^p$ (which is $p\vi(P)+d$) or that of $\psi_{p^2}(p)$ (which is $p^2\vi(P)$).
\end{proof}
\begin{remark}
    Notice that when $\psi_p(p) = 0$, then $\vi(pP) = p^2\vi(P)$ regardless of $\vi(P)$.
    This case is the easiest exceptional case, as we know precisely the value of $\vi(pP)$.
\end{remark}
\begin{lemma}
    \label{lem:except_stessa_trj}
    Let $P, Q \in E^\infty$ be two points such that $\vi(P) = \vi(Q)$. It holds $\trj(P) = \trj(Q)$.
\end{lemma}
\begin{proof}
    Thanks to Lemma \ref{lem:except_jump}, the trajectory of a point is only determined by its minimal degree. Since $\vi(P) = \vi(Q)$, the thesis follows.
\end{proof}
\begin{lemma}
    \label{lem:except_trj_diverse}
    Let $P, Q \in E^\infty$ be two points with $\vi(P) < \vi(Q)$. Then
    \begin{itemize}
        \item if $\vi(Q) \in \trj(P)$, then $\trj(Q) \subseteq \trj(P)$;
        \item otherwise $\trj(P) \cap \trj(Q) = \emptyset$.
    \end{itemize}
\end{lemma}
\begin{proof}
    The first part is a trivial application of Lemma \ref{lem:except_stessa_trj}.
    If $\vi(Q) \not\in \trj(P)$, let us assume by contradiction that $\trj(P) \cap \trj(Q) \neq \emptyset$.
    This means that there are two points, multiples of $P$ and $Q$ respectively (we take without loss of generality $P$ and $Q$ themselves) such that $\vi(P) \neq \vi(Q)$ but $\vi(pP) = \vi(pQ)$.
    By Lemma \ref{lem:except_jump}, there are four possibilities:
    \begin{itemize}
        \item $p^2\vi(P) = p^2\vi(Q)$;
        \item $p\vi(P) + d = p\vi(Q) + d$;
        \item $p^2\vi(P) = p\vi(Q) +d $ or $p^2\vi(Q) = p\vi(P) +d $.
    \end{itemize}
    The first two are clearly impossible since $\vi(P) \neq \vi(Q)$. The latter are symmetric, so let us assume $p^2\vi(P) = p\vi(Q) +d$. It is clear that in this case $p|d$, hence we can write $p\vi(P) = \vi(Q) + h$ where $ph = d$. Moreover, Lemma \ref{lem:except_jump} implies the existence of a solution to
    \[
    \begin{cases}
        p\vi(P) = \vi(Q) + h, \\
        p\vi(P) < \vi(P) + h.
    \end{cases}
    \]
    The above system implies $\vi(Q)<\vi(P)$, which contradicts the hypothesis $\vi(P)<\vi(Q)$. 
\end{proof}
From now on, we will employ the notation of Proposition \ref{prop:indipendenza}, i.e.
\[ g_{nm} = \big(\gm_n\eps^m:1:\ff(\gm_n\eps^m)\big) \in E^\infty \]
where $\{\gm_n\}_{1 \leq n \leq e}$ is an $\F_p$-basis of $\F_q$.
\begin{defi}
    We denote the set
    \[\acal = \{ 1 \leq m \leq k-1 | \trj(g_{1m}) \cap \trj(g_{1i}) = \emptyset \ \forall 1 < i < m \}, \]
    Moreover, for every $m \in \acal$ we define
    \[l_m = \#\trj(g_{1m}).\]
\end{defi}
\begin{remark}
    For any $1 \leq n \leq e$, we could have chosen $g_{nm}$ instead of $g_{1m}$. Indeed, the choice of $n$ is irrelevant by Lemma \ref{lem:except_stessa_trj}, since $\vi(g_{nm}) = m$ for every $n$. 
\end{remark}
\begin{lemma}
    \label{lem:except_lm}
    With the above notation, we have
    \begin{itemize}
        \item $\ord(g_{nm}) = p^{l_m}$ for every $1 \leq n \leq e$;
        \item $\sum_{m \in \acal} l_m = k-1$.
    \end{itemize}
\end{lemma}
\begin{proof}
    By definition the trajectory of $g_{nm}$ has $l_m$ elements, so the first power of $p$ that annihilates it is $p^{l_m}$. Since $E^\infty$ is a $p$-group, then $p^{l_m}$ is the order of $g_{nm}$.
    
    As for the second part, we notice that the definition of $\acal$ together with Lemma \ref{lem:except_trj_diverse} imply that for every $1 \leq n \leq e$, we have the disjoint union
    \[
        \bigsqcup_{m \in \acal} \trj(g_{nm}) = \{1,\dots,k-1\}.
    \]
    Since by definition of $l_m = \#\trj(g_{nm})$, then the thesis follows.
\end{proof}
\begin{prop}
    \label{prop:except_indip}
    For any given $1 \leq m \leq k-1$, the points
    \[ g_{nm} = \big(\gm_n\eps^m:1:\ff(\gm_n\eps^m)\big) \in E^\infty \]
    are linearly independent.
    Moreover, the trajectory of every linear combination of the $\{g_{nm}\}_{1 \leq n \leq e}$ is a subset of $\trj(g_{1m})$.   
\end{prop}
\begin{proof}
    We can follow the same argument of Proposition \ref{prop:indipendenza}. Notice that the minimum degree of $(pP)_x$ can now be determined by either the degree $d$ term of $\psi_{p}(p)$ or $\psi_{p^2}(p)$. In both cases, the coefficient multiplying the point lies in $R_k^*$.
\end{proof}
\begin{teo}
    \label{teo:grp_struct_exceptional}
    Let $E$ be an elliptic curve defined over $R_k$, such that $\vi\big(\psi_p(p)\big) = d > 0$ and the conditions \ref{c1}, \ref{c2} and \ref{c3} are satisfied. Then we have the group isomorphism
    \[ 
    E^\infty \cong \Prod_{m \in \acal} \left( \Z_{p^{l_m}} \right)^e.
    \]
\end{teo}
\begin{proof}
    We can follow again the proof of Theorem \ref{teo:grp_struct_c1}, with the $\{g_{nm}\}_{m \in \acal, 1 \leq n \leq e}$ as generators.
    Proposition \ref{prop:except_indip} gives us the independence, while Lemma \ref{lem:except_lm} gives us both the structure of the group they generate and the counting argument to show that it is the whole $E^\infty$.
\end{proof}
\begin{remark}
    In the main case (Section \ref{subs:main_case}), condition \ref{c1} is satisfied by $\psi_p(p)$ instead of $\psi_{p^2}(p)$, so the minimal degree may only arise from $i \leq p$.
    In this setting, Condition \ref{c2} holds thanks to Corollary \ref{cor:psi_i_zero}, while condition \ref{c3} becomes trivial.
\end{remark}
\begin{prop}
    \label{prop:conditions_23}
    Let $p \in \{2,3\}$ and $E$ be an elliptic curve over $R_k$ such that $\psi_p(p) \not\in R_k^*$.
    Then conditions \ref{c1}, \ref{c2} and \ref{c3} are always satisfied.
\end{prop}
\begin{proof}
    Direct computation (see Lemma \ref{lem:psi_dt_23_calc}). 
\end{proof}
\begin{remark}
    By Proposition \ref{prop:conditions_23}, when $p \in \{2,3\}$ the group structure of $E^\infty$ is given by either Theorem \ref{teo:grp_struct_c1} or Theorem \ref{teo:grp_struct_exceptional}.
    For all the other $p$, we may work without loss of generality with the short Weierstrass forms, i.e. $a_1 = a_2 = a_3 = 0$.
    A direct computation (Lemma \ref{lem:psi_dt_23_calc}) shows that conditions \ref{c1}, \ref{c2} and \ref{c3} are actually satisfied for every $p \leq 13$.
\end{remark}

\subsection{ECDLP} \label{subs:ecdlp}

Given the coordinates of a point $P \in E$ and those of its multiple $Q = nP$, the discrete logarithm problem (DLP) amounts to efficiently recover such $n \in \Z$.
The supposed high complexity of the discrete logarithm problem over elliptic curves is the underlying assumption of several cryptographic protocols.

From the results of the current paper, we efficiently solve the DLP of points in $E^{\infty}$ over $R_k$.
In fact, we consider the base-$p$ representation $n = b_0 + b_1p + \dots + b_{k-1}p^{k-1}$.
Let $m_i = \vi(p^iP)$.
By a repeated application of Lemma \ref{lem:samedegsum} and relation \eqref{eq:except_jump}, we have
\begin{equation}
\label{eq:iterDLP}
    b_i = \left( \Big(Q - \sum_{j=1}^{i-1} b_jp^jP\Big)_x \bmod \eps^{m_i+1}\right) \bigg/ \big((p^iP)_x \bmod \eps^{m_i+1}\big).
\end{equation}
The time complexity of this algorithm is logarithmic in the considered parameters.
More precisely, it has a time complexity of $\log(p)\log(n)$ (see Section \ref{sec:ecdlp} and Example \ref{ex:dlp_comp_1} in the appendix).

The above procedure reduces in polynomial time the discrete logarithm problem of elliptic curves over $R_k$ to the corresponding problem over $\F_q$.
Under the condition \eqref{eq:CoprimalityCondition}, one can lift the DLP from $\F_q$ to $R_k$ in polynomial time, thus the two DLP problems are polynomially equivalent.

Since the point operations over $R_k$ are more expensive than over $\F_q$, the above discussion shows that these elliptic curves are not optimal for DLP-based cryptographic protocols.	

\newpage
\appendix
\section{Explicit computations}
In this section, we give further details and results on some computations omitted for brevity in the paper. Most of the results are obtained using MAGMA \cite{bosma97magma}. All the referenced source code can be found at \cite{repo}.

\subsection{Efficient addition law}
\begin{proposition}
    \label{prop:shortsum_calc}
    Let us follow the notation of Proposition \ref{prop:shortsum}. The coefficients such that
    \begin{equation*}
        \begin{gathered}
            X_1 = g_1H_1 + g_2H_2, \quad Z_3 = g_1H_3 + g_2H_4, \quad Y_3 = H_1H_4 - H_2H_3
        \end{gathered}
    \end{equation*}
    are
    \begin{align*}
        &\begin{aligned}
            H_1 = & -a_1a_3X_1Z_2 - a_3^2Z_1Z_2 + a_1X_2Y_1 - a_2X_1X_2 - a_4X_2Z_1 - a_4X_1Z_2 - 3a_6Z_1Z_2 + Y_1Y_2, \\ 
            H_2 = & -a_2a_3^2Z_1Z_2 + a_1a_3a_4Z_1Z_2 - a_1^2a_6Z_1Z_2 - a_3^2X_1Z_2 + a_4^2Z_1Z_2 - 4a_2a_6Z_1Z_2 + a_3X_2Y_1 \\
                & - a_4X_1X_2 - 3a_6X_2Z_1 - 3a_6X_1Z_2, \\
            H_3 = & \ a_1^2X_1Z_2 + a_1a_3Z_1Z_2 + a_1Y_1Z_2 + a_2X_2Z_1 + a_2X_1Z_2 + a_4Z_1Z_2 + 3X_1X_2, \\
            H_4 = & \ a_1a_3X_1Z_2 + a_3^2Z_1Z_2 + a_2X_1X_2 + a_3Y_1Z_2 + a_4X_2Z_1 + a_4X_1Z_2 + 3a_6Z_1Z_2 + Y_1Y_2.
        \end{aligned}     
    \end{align*}
\end{proposition}
\begin{proof}
    The explicit verification of the above equalities in Magma may be found in the file \code{short\_sum\_verification.magma}.
\end{proof}

\subsection{The polynomial \texorpdfstring{$\ff$}{f}}
\begin{prop}
    \label{prop:zfx_calc}
    Let $\ff \in \rcal[x]$ be the polynomial as defined in Proposition \ref{prop:zfx}. We have
    \begin{align*}
        \begin{aligned}
            \ff(x) \equiv& \ x^3 - a_1x^4 + (a_1^2 + a_2) x^5 - (a_1^3 + 2a_1a_2 + a_3)x^6 +\\
            & (a_1^4 + 3a_1^2a_2 + 3a_1a_3 + a_2^2 + a_4)x^7 - (a_1^5 - 4a_1^3a_2 - 6a_1^2a_3 - 3a_1a_2^2 - 3a_1a_4 - 3a_2a_3)x^8 + \\
            & (a_1^6 + 5a_1^4a_2 + 10a_1^3a_3 + 6a_1^2a_2^2 + 6a_1^2a_4 + 12a_1a_2a_3 + a_2^3 + 3a_2a_4 + 2a_3^2 + a_6)x^9 \bmod x^{10}.
        \end{aligned}
    \end{align*}
\end{prop}
\begin{proof}
    The explicit expression for $\ff$ is obtained by applying the method described in the proof of Proposition \ref{prop:zfx}. Python scripts implementing that logic can be found in the files \code{zfx\_fast.py} (for the extended form) and \code{zfxred\_fast.py} (for the short form).
\end{proof}
\begin{remark}
Values of $\ff(x)$ when $k=30$ (extended form) and $k = 300$ (short form) may be found in \code{zfx\_stored\_30.magma} and \code{zfxred\_stored\_300.magma}, respectively.
We remark that these scripts can handle even $k=150$ (extended) and $k=2000$ (short) in a reasonable time on a commercial laptop.
However, the main purpose of these polynomials is to be employed for computing $\psi_i(n)$ (see below).
For this reason, the values $k=30$ and $k=300$ are considered enough.
\end{remark}
We observe that, when $\chr(\F_q) \nmid 6$, we can represent an elliptic curve in its short Weierstrass form, with $a_1=a_2=a_3=0$, $a_4=A$ and $a_6=B$.
With this notation, we obtain the same result as \cite[Prop. 11]{sala20znz}, i.e.
    $$\ff(x) \equiv x^3 + Ax^7 + Bx^9 \bmod x^{10}.$$

\subsection{Semi-linearity of the sum}
\begin{prop}
    \label{prop:sommemodulop2}
    Let $P_1 = (X_1:1:Z_1),P_2 = (X_2:1:Z_2) \in E^\infty$ be two points and $P_3 = P_1 +_{(0:1:0)} P_2$. Let also $I_P = \langle X_1^2, Z_1 \rangle$. Then
    \begin{align*}
        \begin{aligned}
            (P_3)_x \equiv \ & X_1 + X_2 + \big(a_1X_2 - a_2 X_2^2 + 2a_3 Z_2 - 2a_4 X_2Z_2 - 3a_6 Z_2^2\big)X_1 \bmod I_P.
        \end{aligned}     
    \end{align*}
\end{prop}
\begin{proof}
    Notice that, in accordance with previous notations, $(P_3)_x$ is the $x$ coordinate of $P_3$ in standard form while $X_3$ and $Y_3$ are the results of the addition as defined in \cite{bosma95}. By Proposition \ref{prop:shortsum} we have
    \begin{align*}
        \begin{aligned}
            g_1 &\equiv X_1 + X_2 + a_1X_1X_2 + a_3X_1Z_2 &\bmod I_P,  \\ 
            g_2 &\equiv Z_2 &\bmod I_P,
        \end{aligned}
    \end{align*}
    while Proposition \ref{prop:shortsum_calc} gives
    \begin{align*}
        \begin{aligned}
            H_1 &\equiv 1 - a_2X_1X_2 - a_4Z_2X_1 + a_1X_1 &\bmod I_P, \\
            H_2 &\equiv -a_4X_1X_2 - 3a_6X_1Z_2 + a_3X_1 &\bmod I_P, \\
            H_3 &\equiv 3X_1X_2 + a_2X_1Z_2 &\bmod I_P, \\
            H_4 &\equiv 1 + a_2X_1X_2 + a_4X_1Z_2 &\bmod I_P.
        \end{aligned}     
    \end{align*}
    Computing the addition we obtain
    \begin{align*}
        \begin{aligned}
            X_3 &\equiv X_1 + X_2 - a_2X_2^2X_1 - 2a_4X_1X_2Z_2 + 2a_1X_1X_2 + 2a_3X_1Z_2 - 3a_6X_1Z_2^2 &\bmod I_P, \\
            Y_3 &\equiv 1 + a_1X_1 &\bmod I_P.
        \end{aligned}     
    \end{align*}
    From here it is clear that the inverse modulo $I_P$ of $Y_3$ is $1-a_1X_1$. Multiplying it by $X_3$ we obtain the required expression for $(P_3)_x$. Explicit computations can be found in \code{inspection.magma}.
\end{proof}

\begin{prop}
    \label{prop:inspection_eps}
    Let $P_1 = (X_1:1:Z_1), P_2 = (X_2:1:Z_2) \in E^\infty$ be two points and $P_3 = P_1 +_{(0:1:0)} P_2$.
    Let also $m = \min\{\vi(P_1), \vi(P_2) \}$.
    Then we have
    $$ (P_3)_x \equiv X_1 + X_2 \bmod \eps^{m+1}.$$
\end{prop}
\begin{proof}
    It follows from Proposition \ref{prop:sommemodulop2}, by noting that $\eps^{m+1} | X_1^2$ and $\eps^{m+1} | Z_1$, hence $I_P \subseteq \langle \eps^{m+1} \rangle$ and $\eps^{m+1}$ divides both $X_1X_2$ and $X_1Z_2$, so the third term of the expression of $(P_3)_x$ given by Proposition \ref{prop:sommemodulop2} is $0 \bmod \eps^{m+1}$.
\end{proof}

\subsection{Computing \texorpdfstring{$\psi_i(n)$}{psii(n)} } \label{sec:coeff}
Let $P = (X:1:\ff(X)) \in E^{\infty}$.
Thanks to Proposition \ref{prop:zfx_calc} we know the symbolic expression of $\ff$ for a fixed $k$. The $i$-th multiplication polynomial $\psi_i(n)$ is the coefficient of $X^i$ in $(nP)_x$, which by Theorem \ref{teo:psi_poli} it is a polynomial of degree $i$ in $n$. 

To compute it we derive the symbolic expression of $nP$ as a function of $X$ for $1 \leq n \leq i+1$, and for every monomial expression in $a_1, \dots, a_6$, we fit a degree-$i$ polynomial on its coefficients. This is a standard technique for Faulhaber formulas, see \cite{gunderson10}.
Once we have the interpolated expression for $\psi_i(n)$, we can also validate it: we just check that $(n-1)P + P = np$, where $(n-1)P$ and $nP$ are expressed through the previously computed $\psi_i(n)$.
We remark that on all the computations regarding $\psi_i(n)$ we may assume $k=i+1$.

A script performing this operation can be found in \code{ind.magma} and \code{indred.magma}. The parameter $k$ can be modified, while setting \code{proof=true;} verifies every computed $\psi_i$. The first few values of $\psi_i$ in the extended form can be found in \code{psi\_stored\_30.magma}. The first ones, rearranged, are reported below:
\begin{align*}
    \begin{aligned}
        \psi_1 &= n, \\ 
        \psi_2 &= \binom{n}{2}a_1, \\
        \psi_3 &= \binom{n}{3}a_1^2 - 2\binom{n+1}{3}a_2, \\
        \psi_4 &= \binom{n}{4}a_1^3 - \binom{n+1}{3}(2n-3)a_1a_2 + \frac{n(n^3-1)}{2}a_3. \\
    \end{aligned}     
\end{align*}
The first few values for the short Weirstrass form can be found in \code{psired\_stored\_222.magma}.
The first nonzero values are reported below:
\begin{align*}
    \begin{aligned}
        \psi_1 &= n, \\ 
        \psi_5 &= -\frac{2}{5} An(n^4 - 1), \\
        \psi_7 &= -\frac{3}{7} Bn(n^6 - 1), \\
        \psi_9 &= \frac{2}{15} A^2 (n^4-1)(n^4 - 5). \\
    \end{aligned}     
\end{align*}

\subsection{The exceptional case} \label{sec:exeption_prob}
Due to the structure of our rings $R_k$, a curve $E$ is exceptional if and only if $\pi\big(\psi_p(p)\big) = 0$.
For this reason, for every prime $p$ we count the zeros of $\psi_p(p)$ among the non-singular choices of the curve coefficients.
For computational reasons, we assume $p>5$ so we work with the short Weierstrass form.
The code can be found in \code{except\_coeff.magma}, while results for primes up to $p=79$ can be found in Table \ref{tbl:rates}.
\begin{table}[ht]
    \begin{tabular}{l|l||l|l||l|l||l|l}
    $p$ & rate & $p$ & rate & $p$ & rate & $p$ & rate \\ \hline
    5   & 1/5  & 7   & 1/7  & 11  & 2/11 & 13  & 1/13 \\
    17  & 2/17 & 19  & 2/19 & 23  & 3/23 & 29  & 3/29 \\
    31  & 3/31 & 37  & 1/37 & 41  & 4/41 & 43  & 2/43 \\
    47  & 5/47 & 53  & 3/53 & 59  & 6/59 & 61  & 3/61 \\
    67  & 2/67 & 71  & 7/71 & 73  & 2/73 & 79  & 5/79
    \end{tabular}
    \vspace{0.2cm}
    \caption{Rate of exceptional cases.\label{tbl:rates}}
\end{table}

\begin{example}
    \label{ex:strange_ex1_calc}
    Let $E$ be the elliptic curve over $R_k$ as defined in Example \ref{ex:strange_ex1}.
    In \code{ex1.magma} we perform some computations to verify our claims. First of all, we directly compute the coefficients $\psi_3(3)$ and $\psi_9(3)$, which as shown in Theorem \ref{teo:grp_struct_exceptional} are enough to determine the group structure.
    After that, we pick random points and show how their minimal degree changes when multiplied by $3$.
    Finally, we compute the trajectory of $g_m = \big(\eps^m:1:\ff(\eps^m)\big)$ for every $m \in \acal$, namely every $m$ not contained in the trajectory of any $g_n$ for $n < m$.
\end{example}
\begin{lemma}
    \label{lem:psi_dt_23_calc}
    Let $p \leq 13$ and $E$ be an elliptic curve defined over $R_k$, such that $\vi\big(\psi_p(p)\big) = d > 0$. Then conditions \ref{c1}, \ref{c2}, and \ref{c3} of Theorem \ref{teo:grp_struct_exceptional} are always satisfied.
\end{lemma}
\begin{proof}
    The cases $p \in \{2,3\}$ can be found in \code{proof\_23.magma}. Conditions \ref{c2} and \ref{c3} can be directly computed. For condition \ref{c1} is enough to check that for both $p=2,3$ we symbolically have $\Dt_E \in \langle \psi_p(p), \psi_{p^2}(p)\rangle$ (where $\Dt_E$ is the discriminant of the curve $E$), and since $\psi_p(p) \not\in R_k^*$ and $\Dt_E \in R_k^*$, we must have $\psi_{p^2}(p) \in R_k^*$. \\
    Notice that for $5 \leq p \leq 13$ we can restrict ourselves to work with a short Weierstrass form. Again \ref{c2} and \ref{c3} are directly computed. \ref{c1} is similarly implied by verifying that
    \[ \Dt_E \in \langle A, B \rangle \subseteq \sqrt{\langle \psi_p(p), \psi_{p^2}(p)\rangle}, \]
    thus $\psi_p(p)$ and $\psi_{p^2}(p)$ cannot both belong to $\m$. \\
    Explicit computations can be found in \code{proof\_short.magma}.
\end{proof}

\subsection{ECDLP} \label{sec:ecdlp}
We provide here further details about the algorithm introduced in Section \ref{subs:ecdlp}. As already shown, we must compute
\[ b_i = \left( \Big(Q - \sum_{j=1}^{i-1} b_jp^jP\Big)_x \bmod \eps^{m_i+1}\right) \bigg/ \big((p^iP)_x \bmod \eps^{m_i+1}\big). \]
To entirely determine $n$, we compute the $\log_p(n)$ digits $b_i$ of its base-$p$ representation.
If we store the partial sum $S_i = \sum_{j=1}^{i} b_jp^jP$ and the point $p^iP$, at every step we can compute $b_i$ using only two point multiplication, since $S_{i+1} = S_{i} + b_{i+1}p^{i+1}P$ and $p^{i+1}P = p(p^{i}P)$.
We multiply by $b_i<p$ and $p$ respectively, so these multiplication have complexity $\log(p)$.
The inversion is a field operation, so the whole algorithm has complexity $\log(p)\log(n)$.

\begin{example}
    \label{ex:dlp_comp_1}
    Let us consider $q=p=2$ and $k=10$, namely the ring $\rcal=\F_2[x]/(x^{10}) \simeq \F_2[\eps]$, and the curve $E$ defined over $\rcal$ by
    $a_1 = \eps^8 + \eps^7 + \eps^6 + \eps^3 + \eps^2 + 1$, 
    $a_2 = \eps^9 + \eps^8 + \eps^6 + \eps^4 + \eps^3 + \eps^2 + \eps + 1$, 
    $a_3 = \eps^9 + \eps^8 + \eps^4 + \eps^3 + \eps + 1$, 
    $a_4 = \eps^7 + \eps^4 + \eps^2$ and 
    $a_6 = \eps^9 + \eps^7 + \eps^5 + \eps^3$. 
    We address the DLP given by $Q = nP$, with
    \[
    P_x = \eps^9 + \eps^7 + \eps^6 + \eps^3 + \eps, \quad Q_x = \eps^6 + \eps^5 + \eps^4 + \eps^3 + \eps.
    \]
    At the first iteration, $m_0 = \vi(P) = 1$. By \eqref{eq:iterDLP} we compute $b_0 = 1$ and
    $(Q-S_0)_x = \eps^7 + \eps^6 + \eps^4$.
    At the second iteration, $m_1 = 2$ and we obtain $b_1 = 0$. Hence we do not update the sum and get $(Q-S_1)_x = (Q-S_0)_x$. Similarly, we obtain $b_2 = b_3 = 1$, which imply $n =$ \texttt{1101}$_\text{2}$ $ = 13$. 
\end{example}
\begin{example}
    \label{ex:dlp_comp_2}
    Let us consider $q=p=3$ and $k=29$, namely the ring $\rcal=\F_3[x]/(x^{29}) \simeq \F_3[\eps]$, and the curve $E$ defined over $\rcal$ by
    \begin{equation*}
        \begin{gathered}
            a_1 = 2\eps^{27} + \eps^{26} + 2\eps^{23} + 2\eps^{22} +\eps^{21} + 2\eps^{20} +\eps^{19} + 2\eps^{16} + \\
            2\eps^{13} + \eps^{10} + 2\eps^9 + 2\eps^8 +\eps^6 +\eps^4 + 2\eps^3 + 2\eps^2 +\eps + 2,
        \end{gathered}
    \end{equation*}
    \begin{equation*}
        \begin{gathered}
            a_2 =\eps^{27} + 2\eps^{25} +\eps^{24} + 2\eps^{22} +\eps^{19} +\eps^{18} +\eps^{16} +\eps^{15} +\eps^{13} + 2\eps^{12} \\
            + 2\eps^{11} +\eps^{10} +\eps^9 +\eps^8 + 2\eps^7 + 2\eps^6 + 2\eps^4 + 2\eps^3 +\eps^2 +\eps,
        \end{gathered}
    \end{equation*}
    \begin{equation*}
        \begin{gathered}
            a_3 = 2\eps^{28} +\eps^{27} +\eps^{26} + 2\eps^{24} +\eps^{23} +\eps^{22} +\eps^{21} +\eps^{20} +\eps^{19} + 2\eps^{17} \\ 
            + \eps^{16} +\eps^{14} +\eps^{13} +\eps^{12} +\eps^{11} +\eps^8 +\eps^6 +\eps^4 +\eps^2 + 1,
        \end{gathered}
    \end{equation*}
    \begin{equation*}
        \begin{gathered}
            a_4 =\eps^{28} +\eps^{27} + 2\eps^{22} + 2\eps^{21} + 2\eps^{20} +\eps^{19} +\eps^{18} +\eps^{17} +\eps^{15} +\eps^{14} \\ 
            + \eps^{13} + 2\eps^{12} + 2\eps^{11} +\eps^9 + 2\eps^8 + 2\eps^7 +\eps^6 + 2\eps^5 + 2\eps^4 + 2,
        \end{gathered}
    \end{equation*}
    \begin{equation*}
        \begin{gathered}
            a_6 = 2\eps^{28} + 2\eps^{27} + 2\eps^{26} +\eps^{25} +\eps^{24} + 2\eps^{23} + \eps^{22} +\\
            \eps^{21} + 2\eps^{19} + \eps^{18} + 2\eps^{17} + 2\eps^{16} +\eps^{15} + \eps^{13} + \\
             2\eps^{11} +\eps^{10} +\eps^8 + 2\eps^7 + 2\eps^6 + 2\eps^4 + 2\eps^3 +\eps^2 + 1.
        \end{gathered}
    \end{equation*}
    We address the DLP given by $Q = nP$, with
    \begin{equation*}
        \begin{gathered}
            P_x = 2\eps^{28} + 2\eps^{24} + 2\eps^{23} + 2\eps^{22} +\eps^{21} +\eps^{20} + 2\eps^{18} + \eps^{17} + \\
            2\eps^{16} + 2\eps^{15} +\eps^{14} + 2\eps^{13} +\eps^{12} +\eps^8 +\eps^7 + 2\eps^2 + 2\eps
        \end{gathered}
    \end{equation*}
    and
    \begin{equation*}
        \begin{gathered}
            Q_x =\eps^{27} + 2\eps^{25} +\eps^{24} + 2\eps^{23} + 2\eps^{22} + 2\eps^{20} +\eps^{19} + 2\eps^{18}  \\
            +\eps^{15} +\eps^{12} + 2\eps^{11} +\eps^{10} + 2\eps^9 + 2\eps^8 + 2\eps^6 + 2\eps^5 + 2\eps^3 +\eps
        \end{gathered}
    \end{equation*}
    At the first iteration, $m_0 = \vi(P) = 1$. We compute $b_0 = 2$ and the residual
    \begin{equation*}
        \begin{gathered}
            (Q-S_0)_x = 2\eps^{28} + 2\eps^{27} + 2\eps^{25} +\eps^{24} + 2\eps^{21} +\eps^{20} +\eps^{18} \\ +\eps^{17} +\eps^{16} + 2\eps^{15} + 2\eps^{12} + 2\eps^{11} +\eps^{10} + 2\eps^9
        \end{gathered}
    \end{equation*}
    At the second iteration, $m_1 = 3$ and we obtain $b_1 = 0$. Hence we do not update the sum and get $(Q-S_1)_x = (Q-S_0)_x$. Similarly, we obtain $b_2 = b_3 = 1$, which imply $n =$ \texttt{1102}$_\text{3}$ $ = 38$. 
\end{example}
The implemented algorithm, as well as the actual recovery of $n$ for random choices of $E$ and $p$, can be found in \code{d\_log.magma}.
	
\bibliographystyle{plain}
\bibliography{refs}
	
\end{document}